\documentclass[12pt,twoside]{amsart}

\usepackage{amsmath}
\usepackage{amsfonts}
\usepackage{amssymb,enumerate}
\usepackage{amsthm}
\usepackage[all]{xy}
\usepackage{hyperref}
\allowdisplaybreaks[4] \footskip=12pt
\renewcommand{\uppercasenonmath}[1]{}

\numberwithin{equation}{section} \theoremstyle{plain}
\newtheorem*{thm*}{Main Theorem}
\newtheorem{thm}{Theorem}[section]
\newtheorem{cor}[thm]{Corollary}
\newtheorem*{cor*}{Corollary}
\newtheorem{lem}[thm]{Lemma}
\newtheorem*{lem*}{Lemma}

\newtheorem*{fact*}{Fact}

\newtheorem*{nota*}{Notation}

\newtheorem*{prop*}{Proposition}
\newtheorem{rem}[thm]{Remark}
\newtheorem*{rem*}{Remark}

\newtheorem*{observation*}{Observation}
\newtheorem{exa}[thm]{Example}
\newtheorem*{exa*}{Example}
\newtheorem{df}[thm]{Definition}
\newtheorem*{df*}{Definition}

\newtheorem*{conj*}{Conjecture}

\newtheorem*{con*}{Construction}

\renewcommand{\geq}{\geqslant}
\renewcommand{\leq}{\leqslant}

\begin{document}
\begin{center}
{\large  \bf  Flat and Ghost dimensions in tensor triangulated categories with a compact tilting generator}

\vspace{0.5cm}  Xiaoyan Yang\\
School of Science, Zhejiang University of Science and Technology, Hangzhou 310023\\
E-mail: yangxy@zust.edu.cn
\end{center}

\bigskip
\centerline { \bf  Abstract}
\leftskip10truemm \rightskip10truemm \noindent  Let $\mathcal{T}$ be a tensor triangulated category with a compact tilting generator $M$ and abelian heart $\mathcal{H}_M = M^{\bot_{\neq0}}$.
We define a triangulated flat dimension $\mathrm{flatdim}_{\mathcal{T}}X$ for $X\in\mathcal{T}$, and give criteria in terms of tensor vanishing and resolutions by the class of pure quotients of $\mathrm{Add}(M)$.
We further introduce $M$-ghost length and $M$-ghost dimension, and show that the global dimension of
$\mathcal{H}_M$ controls the $M$-ghost length of all objects in $\mathcal{T}$, and that the $M$-ghost dimension of $\mathcal{T}$ coincides with the weak global dimension of $\mathcal{H}_M$ and the supremum of $\mathrm{flatdim}_{\mathcal{T}}X+\mathrm{inf}X$ of bounded objects, and the relative triangulated dimension of
$\mathcal{T}^c$ with respect to $\mathrm{add}(M)$.
This extends the classical Hovey-Lockridge theorem, which is recovered for derived categories of rings.\\
\vbox to 0.3cm{}

\leftskip0truemm \rightskip0truemm
\bigskip
\section* { \bf Introduction}

Triangulated categories arise naturally in many areas of modern algebra and geometry, including derived categories, stable homotopy theory, and representation theory. A central challenge is to import the tools of classical homological algebra into this setting, where abelian structures are absent. Tilting (or silting) objects provide a bridge: a tilting object $M$ in a triangulated category $\mathcal{T}$ induces an abelian heart $M^{\bot_{\neq0}}=\{Y\in\mathcal{T}\hspace{0.02cm}|\hspace{0.02cm}\mathrm{Hom}_\mathcal{T}(M,Y[i])=0\ \textrm{for}\ i\neq0\}$, which inherits much of the homological information from $\mathcal{T}$.
This allows one to translate problems in triangulated categories into problems in module categories or more general abelian categories, where tools like projective dimension and global dimension are available (see, e.g., \cite{AGMSV,BCRPZ,CCZ,HKM,Ko,N0,Y1}).

In a triangulated category equipped with a distinguished object (like the ring
$R$ in its derived category $\mathrm{D}(R)$), a morphism is called a \emph{ghost}  if it becomes zero after applying the representable functor $\mathrm{Hom}_{\mathrm{D}(R)}(R[i],-)$, i.e., if it is invisible to the cohomology theory represented by $R$. Ghosts form an ideal, and their composites define a filtration: an object $X$ has ghost length at most
$n$ if every composite of
$n+1$ ghosts out of
$X$ vanishes. Taking the supremum over perfect complexes gives the ghost dimension
$\mathrm{gh.dim}R$ of $R$. Hovey and Lockridge \cite{H} established the fundamental identification
\begin{center}$\mathrm{gh.dim}R=\mathrm{wgldim}R$,\end{center}where
$\mathrm{wgldim}R$ denotes the classical weak global dimension. This theorem provides a striking link between a purely triangulated notion and a classical ring-theoretic invariant.
 It has deep roots in stable homotopy theory: the generating hypothesis asks whether every ghost between finite spectra is zero, and ghost dimension precisely measures the failure of this hypothesis. Moreover, the invariant extends to ring spectra, where classical module categories are unavailable, and recovers the expected weak dimension for Eilenberg-Mac Lane spectra, confirming its topological relevance, see \cite{HL,H,HLP,L}.

The key to the Hovey-Lockridge theorem is flatness: the weak global dimension on the right-hand side is defined via flat resolutions, and in the derived category
$\mathrm{D}(R)$,  flat dimension admits a characterization via the vanishing of derived tensor products with arbitrary complexes. This tensor product characterization is precisely what allows one to connect ghost dimension to weak global dimension. Thus flatness is not merely a classical notion but a central ingredient in the triangulated story.

The classical framework, however, relies crucially on the tensor product of modules. When one moves beyond $\mathrm{D}(R)$ to a general compactly generated triangulated category $\mathcal{T}$ with a compact silting object
$M$, this structure is lost. Although $M$ induces a $t$-structure $(\mathrm{Add}(M)^{\bot_{>0}}, \mathrm{Add}(M)^{\bot_{\leq 0}})$ with abelian heart $\mathcal{H}_M=M^{\bot_{\neq0}}$ and a cohomological functor $\mathcal{H}^0_M(-):\mathcal{T}\rightarrow\mathcal{H}_M$, flat dimension is no longer canonically defined. Yang \cite{Y1} has made significant progress by showing that silting objects give rise to sup-projective and inf-injective resolutions, thereby developing a homological dimension theory for
$\mathcal{T}$. While this provides a powerful unified framework, a flat dimension that genuinely mirrors the classical tensor product vanishing criterion has remained undefined.

In this paper, we work in the tensor triangulated setting with tensor unit ${\i}$, where a tensor product is available, and fill this gap by introducing a flat dimension
$\mathrm{flatdim}_{\mathcal{T}}X$ for $X\in\mathcal{T}$,  defined via the vanishing of negative cohomology after tensoring with bounded-below objects. Intuitively, this measures the least shift $n$ for which $U\otimes X$ has no negative cohomology for any bounded-below $U$. Under the hypothesis \(\operatorname{add}(M)=\operatorname{add}({\i})\), we establish its relation to the heart $\mathcal{H}_M$ and to resolutions by the class \(\mathcal{F}\) of pure quotients of \(\mathrm{Add}(M)\) (Theorems \ref{lem2.1} and \ref{lem2.21}), thereby recovering the classical flat dimension characterization for complexes and differential graded (dg) modules as  special cases.

 The resulting notion of
$M$-ghosts, maps $f$ satisfying $\mathrm{Hom}_\mathcal{T}(M[-i], f) = 0$ for all $i \in \mathbb{Z}$, has recently attracted growing interest, as the tilting object provides a bridge between the triangulated category and its heart. This naturally raises the question: how do the $M$-ghost dimension of
$\mathcal{T}$, the (weak) global dimension of $\mathcal{H}_M$, the flat dimension of objects, and the generation of $\mathcal{T}^c$ by $\mathrm{add}(M)$ related?

We address this question by introducing $M$-ghost length and
$M$-ghost dimension. An object $X\in\mathcal{T}$ is said to \emph{have
$M$-ghost length at most $n$}, $\mathrm{gh.len}_MX\leq n$, if every composite
of $n+1$ $M$-ghosts is $0$. The
$M$-ghost dimension of $\mathcal{T}$ is defined as
\begin{center}$\mathrm{gh.dim}\mathcal{T}=
\mathrm{sup}\{\mathrm{gh.len}_MX\hspace{0.03cm}|\hspace{0.03cm}X\in\mathcal{T}^c\}$,\end{center}
where
$\mathcal{T}^c$ denotes the full subcategory of compact objects. Our first  main theorem (Theorem \ref{lem5.2}) shows that, for a tilting object
$M$ and $n\geq0$,
\begin{center}$\mathrm{gldim}\mathcal{H}_M\leq n\Longleftrightarrow\mathrm{gh.len}_MX\leq n\ \textrm{for\ all}\ X\in\mathcal{T}$.\end{center}
Thus the global dimension of the heart
$\mathcal{H}_M$  precisely controls the
$M$-ghost length of every object. The proof proceeds by translating ghost composites in $\mathcal{T}$ into exact sequences in $\mathcal{H}_M$ via the
$t$-structure, and then applying the standard projective resolution criterion for projective dimension.

Following \cite{AAITY},  the \emph{dimension of $\mathcal{T}^c$ with respect to $\mathrm{add}(M)$} is defined as
\begin{center}$\mathrm{add}(M)\textrm{-tri.dim}\mathcal{T}^c:=\textrm{inf}\{n\geq0\hspace{0.02cm}|\hspace{0.02cm}
\mathcal{T}^c=\mathrm{thick}^{n+1}(\mathrm{add}(M))\}$.\end{center} Our second main theorem (Theorem \ref{lem5.5}) establishes the stronger equalities
\begin{center}$\mathrm{sup}\{\mathrm{flatdim}_{\mathcal{T}}X+\mathrm{inf}X\hspace{0.03cm}|\hspace{0.03cm}X\in\mathcal{T}^b\}
=\mathrm{wgldim}\mathcal{H}_M=\mathrm{gh.dim}\mathcal{T}
=\mathrm{add}(M)\textrm{-}\mathrm{tri.dim}\mathcal{T}^c$.\end{center}
The key novelty here is the replacement of projective objects by flat ones: we show that the class \(\mathcal{F}\) plays the role of flat objects in $\mathcal{H}_M$. This result directly parallels the classical Hovey-Lockridge theorem and confirms that our flat dimension is the correct invariant for bridging ghost theory with homological algebra in the triangulated setting.

The paper is organised as follows. Section 1 recalls the necessary background on compactly generated tensor triangulated categories, $t$-structures, duality, and ghost ideals. Section 2 develops the technical machinery of flat dimension, giving vanishing and resolution-theoretic criteria (Theorems \ref{lem2.1} and \ref{lem2.21}). Section 3 introduces
$M$-ghost length and $M$-ghost dimension, and contains the main theorems (Theorems  \ref{lem5.2} and \ref{lem5.5}), which unify the ghost dimension of $\mathcal{T}$, the (weak) global dimension of the heart $\mathcal{H}_M$, the flat dimension of bounded objects, and the relative triangulated dimension of compact objects with respect to
$\operatorname{add}(M)$. Finally, we apply our results to derived categories of rings, thereby recovering classical theorems of Hovey and Lockridge as special cases.

\bigskip
\section{\bf Preliminaries}
This section is devoted to recalling the necessary preliminaries and some key tools, it will be used throughout the sequel.
For terminology we shall follow \cite{BBD,BW,KS,Y1}.

 A triple $(\mathcal{T},[1],\Delta)$ is called a \emph{triangulated category}, where $\mathcal{T}$ is an additive category, $[1]$
is an autoequivalence of $\mathcal{T}$ and $\Delta$ is a full subcategory of Diag$(\mathcal{T},[1])$ which satisfies some axioms, where Diag$(\mathcal{T},[1])$ denotes the category whose objects are diagrams in $\mathcal{T}$ of the form $X\rightarrow Y \rightarrow Z\rightarrow X[1]$. The elements of $\Delta$ are called \emph{exact triangles}.

For full subcategories $\mathcal{X},\mathcal{Y}\subseteq\mathcal{T}$, let $\mathcal{X}\ast\mathcal{Y}$ denote the class of objects $Z$  fitting into an exact triangle
 $X\rightarrow Z\rightarrow Y\rightarrow X[1]$ with $X\in\mathcal{X}$ and $Y\in\mathcal{Y}$. Let $\mathrm{Add}(\mathcal{X})$ (resp. $\mathrm{Prod}(\mathcal{X})$) be the category of summands of
 coproducts (resp. products) of objects in $\mathcal{X}$, and $\mathrm{add}(\mathcal{X})$ the summands of finite coproducts. Define $\mathrm{thick}^0(\mathcal{X})=\mathrm{add}(\emptyset)$, $\mathrm{thick}^1(\mathcal{X})=\mathrm{add}\{X[i]\hspace{0.02cm}|\hspace{0.02cm}X\in\mathcal{X}, i\in\mathbb{Z}\}$,  and inductively $\mathrm{thick}^n(\mathcal{X})=
\mathrm{add}(\mathrm{thick}^1(\mathcal{X})\ast\mathrm{thick}^{n-1}(\mathcal{C}))$. Similarly, set $\langle\mathcal{X}\rangle_0=\{0\}$, $\langle\mathcal{X}\rangle_1=\mathrm{Add}\{X[i]\hspace{0.02cm}|\hspace{0.02cm}X\in\mathcal{X}, i\in\mathbb{Z}\}$, and $\langle\mathcal{X}\rangle_n=\mathrm{Add}(\langle \mathcal{X}\rangle_1\ast\langle \mathcal{X}\rangle_{n-1})$ for $n\geq1$. For an interval $\Theta\subseteq\mathbb{Z}$, write \begin{center}$\mathcal{X}^{\perp_\Theta}:=\{Y\in\mathcal{T}\hspace{0.02cm}|\hspace{0.02cm}\mathrm{Hom}_\mathcal{T}(X,Y[i])=0\ \textrm{for}\ X\in\mathcal{X}\ \textrm{and}\ i\in \Theta\},$\end{center} with the convention that $\perp=\perp_0$.

\vspace{2mm}
{\bf Compactly generated triangulated categories.} A pair $(\mathcal{U},\mathcal{V})$ of subcategories  in a triangulated category $\mathcal{T}$ is said to be a \emph{t-structure} if $\mathcal{U}[1]\subseteq\mathcal{U}$, $\mathcal{U}^\bot=\mathcal{V}, {^\bot}\mathcal{V}=\mathcal{U}$ and
$\mathcal{U}\ast\mathcal{V}=\mathcal{T}$.  The \emph{heart} of a $t$-structure is the abelian category $\mathcal{U}\cap\mathcal{V}[1]$. A subcategory
$\mathcal{M}$ of $\mathcal{T}$ is \emph{silting} if $(\mathcal{M}^{\bot_{>0}}, \mathcal{M}^{\bot_{\leq0}})$ is a $t$-structure. An object $M$ in $\mathcal{T}$ is \emph{silting} if $\mathrm{Add}(M)$ is a silting subcategory. If moreover $\mathrm{Hom}_\mathcal{T}(\mathcal{M},\mathcal{M}[\neq0])=0$, then  $\mathcal{M}$ (or $M$) is
\emph{tilting}.

An object $X$ in a triangulated category $\mathcal{T}$ is said to
be \emph{compact} if $\mathrm{Hom}_\mathcal{T}(X,-)$ commutes with all existing coproducts. Compact objects form a thick subcategory
of $\mathcal{T}$, denoted by $\mathcal{T}^c$. If $\mathcal{T}$ is cocomplete, $\mathcal{T}^c$ is skeletally small and $(\mathcal{T}^c)^\bot=0$, then $\mathcal{T}$ is called \emph{compactly generated}. In this case,
for an object $C$ in $\mathcal{T}^c$, there is a unique object $C^\ast$
of $\mathcal{T}$, up to isomorphism, representing the functor
$$\mathrm{Hom}_\mathbb{Z}(\mathrm{Hom}_\mathcal{T}(C,-),\mathbb{Q}/\mathbb{Z}):
\mathcal{T}^{\mathrm{op}}\rightarrow\mathrm{Ab}$$ We use
$(-)^\ast$
  to denote this Brown-Comenetz dual, it is a contravariant functor on $\mathcal{T}^c$.

Let $M$ be a compact silting generator. Denote $\mathcal{P}=\mathrm{Add}(M)$ and $\mathcal{I}=\mathrm{Prod}(M^\ast)$. For $n\geq0$, put $\mathcal{T}^{\leq n}=\mathcal{P}^{\perp_{>n}}$
and  $\mathcal{T}^{\geq n+1}=\mathcal{P}^{\perp_{\leq n}}$,  these form the $t$-structure induced by $M$. Write $\mathcal{T}^{-}=\bigcup_{n>0}\mathcal{T}^{\leq n},\ \ \mathcal{T}^{+}=\bigcup_{n>0}\mathcal{T}^{\geq -n}$.
The truncation functors
 $\tau^{\leq \ell}$ and $\tau^{\geq \ell}$ associated to this $t$-structure  give the following canonical triangles for
$X\in\mathcal{T}$ and $\ell\in\mathbb{Z}$:
  $$\tau^{\leq \ell}(X)\rightarrow X\rightarrow \tau^{\geq \ell+1}(X)\rightarrow\tau^{\leq \ell}(X)[1],$$
 $$\tau^{\leq \ell-1}(X)\rightarrow \tau^{\leq \ell}(X)\rightarrow\mathcal{H}^{\ell}_M(X)[-\ell]\rightarrow\tau^{\leq \ell-1}(X)[1],$$
$$\tau^{\geq\ell+1}(X)[-1]\rightarrow\mathcal{H}^{\ell}_M(X)[-\ell]\rightarrow\tau^{\geq\ell}(X)\rightarrow \tau^{\geq\ell+1}(X).$$
These triangles will be used throughout to reduce statements about arbitrary objects to statements about their cohomology objects in the heart.

\begin{lem}\label{lem0.2} {\rm (\cite[Lemma 2.5]{Y1})} For $P\in\mathcal{P}$ and $I\in\mathcal{I}$, $X\in\mathcal{T}$ and $\ell\in\mathbb{Z}$, one has $$\mathrm{Hom}_\mathcal{T}(P[-\ell],X)\cong\mathrm{Hom}_{\mathcal{T}}(P,\mathcal{H}^\ell_M(X))
\cong\mathrm{Hom}_{\mathcal{H}_M}(\mathcal{H}^0_M(P),\mathcal{H}^\ell_M(X)),$$
$$\mathrm{Hom}_\mathcal{T}(X,I[-\ell])\cong\mathrm{Hom}_{\mathcal{T}}(\mathcal{H}^\ell_M(X),I)
\cong\mathrm{Hom}_{\mathcal{H}_M}(\mathcal{H}^\ell_M(X),\mathcal{H}^0_M(I)).$$
\end{lem}

This lemma is fundamental: it allows us to compute
$\mathrm{Hom}_\mathcal{T}$ between
$\mathcal{P}$ (or
$\mathcal{I}$) and any object
$X$ entirely in terms of the abelian category $\mathcal{H}_M$.

Following \cite{Mi18,Y1}, for $Z\in\mathcal{T}$ define $$\mathrm{sup}Z=\mathrm{sup}\{\ell\in\mathbb{Z}\hspace{0.03cm}|\hspace{0.03cm}\mathcal{H}^\ell_M(Z)\neq0\}, \ \mathrm{inf}Z=\mathrm{inf}\{\ell\in\mathbb{Z}\hspace{0.03cm}|\hspace{0.03cm}\mathcal{H}^\ell_M(Z)\neq0\}$$
and $\mathrm{amp}Z=\mathrm{sup}Z-\mathrm{inf}Z$.
Let $X\in\mathcal{T}^{-}$. Then $\tau^{\leq s}(X)\cong X$ where $s=\mathrm{sup}X$.  Lemma \ref{lem0.2} yields $\mathrm{Hom}_\mathcal{T}(\bar{P}[-s],X)\cong
\mathrm{Hom}_{\mathcal{H}_M}(\mathcal{H}^0_M(\bar{P}),\mathcal{H}^s_M(X))$ for $\bar{P}\in\mathcal{P}$, hence there is a morphism $f:\bar{P}[-s]\rightarrow X$ inducing a surjection on $\mathcal{H}^s_M$. Such an $f$ is called a \emph{sup-projective (sppj) morphism}. Dually,
for $Y\in\mathcal{T}^{+}$, there exists a morphism $g:Y\rightarrow I$ with $I\in\mathcal{I}[-\mathrm{inf}Y]$ inducing an injection on $\mathcal{H}^{\mathrm{inf}Y}_M$.  such a $g$ is called an \emph{inf-injective (ifj) morphism}.

Following \cite{C}, an exact triangle $X\rightarrow Y\rightarrow Z\stackrel{g}\rightarrow X[1]$ is called \emph{$M$-ghost} if $$0\rightarrow\mathrm{Hom}_\mathcal{T}(M[i],X)\rightarrow\mathrm{Hom}_\mathcal{T}(M[i],Y)\rightarrow\mathrm{Hom}_\mathcal{T}(M[i],Z)\rightarrow0$$ is exact, equivalently if $0\rightarrow\mathrm{Hom}_\mathcal{T}(Z,M^\ast[i])\rightarrow\mathrm{Hom}_\mathcal{T}(Y,M^\ast[i])
\rightarrow\mathrm{Hom}_\mathcal{T}(X,M^\ast[i])\rightarrow0$ is exact for all $i\in\mathbb{Z}$. We call $g$ an \emph{$M$-ghost}, and denote $\mathcal{G}_M$ the ideal of $M$-ghosts, and $\mathcal{G}^n_M$ consists of the $n$-fold composites of $M$-ghost maps for $n\geq1$.

\begin{lem}\label{lem2.0} {\rm (\cite[Lemma 4.3]{Y1})} An exact triangle $X\rightarrow Y\rightarrow Z\rightarrow X[1]$ is $M$-ghost if and only if $0\rightarrow\mathcal{H}^\ell_M(X)\rightarrow\mathcal{H}^\ell_M(Y)\rightarrow\mathcal{H}^\ell_M(Z)\rightarrow0$  is exact in $\mathcal{H}_M$ for all $\ell\in\mathbb{Z}$.
\end{lem}

\begin{lem}\label{lem1.00} Let $X$ be an object in $\mathcal{T}$.

$(1)$ There exists an $M$-ghost triangle $Y\rightarrow P\rightarrow X\rightarrow Y[1]$ such that $P\in\langle\mathcal{P}\rangle_{1}$.

$(2)$ There exists an $M$-ghost triangle $X\rightarrow I\rightarrow Z\rightarrow X[1]$ such that $I\in\langle\mathcal{I}\rangle_{1}$.
\end{lem}
\begin{proof} We just prove the statement (1) since (2) is dual.

By Lemma \ref{lem0.2}, there is a morphism $a_i:M_i\rightarrow\mathcal{H}^i_M(Z)$ with $M_i\in\mathcal{P}$ for $i\in\mathbb{Z}$ such that $\mathcal{H}^0_M(a_i)$ is surjective, it follows that $$\mathrm{Hom}_\mathcal{T}(\coprod_{i\in\mathbb{Z}}M_i[-i],Z)\cong\prod_{i\in\mathbb{Z}}\mathrm{Hom}_\mathcal{T}(M_i[-i],Z)\cong
\prod_{i\in\mathbb{Z}}\mathrm{Hom}_\mathcal{T}(M_i,\mathcal{H}^i_M(Z)).$$As  $\mathcal{H}^\ell_M(\coprod_{i\in\mathbb{Z}}M_i[-i])\rightarrow\mathcal{H}^0_M(M_\ell)$ is surjective for all $\ell\in\mathbb{Z}$, it follows from the above isomorphism and Lemma \ref{lem2.0} that $$Y\rightarrow\coprod_{i\in\mathbb{Z}}M_i[-i]\rightarrow X\rightarrow Y[1]$$ is an $M$-ghost triangle with $P=\coprod_{i\in\mathbb{Z}}M_i[-i]\in\langle\mathcal{P}\rangle_{1}$.
\end{proof}

Following \cite{K1}, a \emph{cone} over a diagram \(\{X_\lambda\}_{\lambda\in\Lambda}\) in \(\mathcal{T}\) is a family \(\{\sigma_\lambda:X_\lambda\to X\}\) compatible with the diagram. A \emph{weak colimit} is a universal cone. A \emph{homology colimit} is a cone such that
$\textrm{colim}\mathrm{Hom}_{\mathcal{T}}(C,X_\lambda)\rightarrow \mathrm{Hom}_{\mathcal{T}}(C,X)$
is an isomorphism for all $C\in\mathcal{T}^c$.

We now recall the tensor structure and the related notion of purity, both of which are essential for the definition of flat dimension.

\vspace{2mm}
{\bf Tensor triangulated categories.} Let $(\mathcal{T},\otimes, {\i})$ be a tensor triangulated category as defined in \cite[\S 8]{BIK}. In particular, $\mathcal{T}$ is a compactly generated triangulated category with a symmetric monoidal
structure, $\otimes$ is its tensor product and ${\i}$ the unit of the tensor product. The tensor
product is exact in each variable and preserves coproducts.

For objects $X,Y,Z$ in $\mathcal{T}$, there is a function object $\mathcal{H}\mathrm{om}(X,Y)$ satisfying
\begin{align}
\mathrm{Hom}_\mathcal{T}(X\otimes Z,Y)\cong\mathrm{Hom}_\mathcal{T}(Z,\mathcal{H}\mathrm{om}(X,Y))
\label{exact03}
\tag{$\ast$}\end{align}by Brown representability.
Moreover, there exists a natural isomorphism
\begin{align}
\mathcal{H}\mathrm{om}(X\otimes Z,Y)\cong\mathcal{H}\mathrm{om}(Z,\mathcal{H}\mathrm{om}(X,Y)).
\label{exact03}
\tag{$\ast\ast$}\end{align}

 Assume that the tensor unit ${\i}$ is compact and that all compact objects
$C$ are \emph{strongly dualizable} in the sense that the canonical morphism
\begin{center}
$C^\vee\otimes X\rightarrow\mathcal{H}\mathrm{om}(C,X)$
\end{center}
is an isomorphism for all $X$ in $\mathcal{T}$, where $X^\vee:=\mathcal{H}\mathrm{om}(X,{\i})$.

Let $C, D$ be compact objects in $\mathcal{T}$. The following properties are easily verified.

(i) ${\i}^\vee\cong{\i}$, $C^{\vee\vee}\cong C$, $C^\vee$ and $C\otimes D$ are compact;

(ii) $\mathrm{Hom}_\mathcal{T}(X\otimes C^\vee,Y)\cong\mathrm{Hom}_\mathcal{T}(X,C\otimes Y)$.

Following \cite{B,K}, we call an exact triangle $X\stackrel{f}\rightarrow Y\stackrel{g}\rightarrow Z\rightarrow X[1]$ \emph{pure} if for all $C\in\mathcal{T}^c$, the induced sequence \begin{center}$0\rightarrow\mathrm{Hom}_{\mathcal{T}}(C,X)\rightarrow\mathrm{Hom}_{\mathcal{T}}(C,Y)\rightarrow
\mathrm{Hom}_{\mathcal{T}}(C,Z)\rightarrow0$\end{center} is exact. In this case, $f$ is a \emph{pure monomorphism} and $g$ is a \emph{pure epimorphism}. An object $X$ in $\mathcal{T}$ is called \emph{pure-injective} if every pure monomorphism
$X\rightarrow Y$ splits.

\begin{df}\label{lem:0.1} {\rm (\cite{BW}) A \emph{duality triple} $(\mathcal{T}, \mathcal{S}, Q)$ consists of two compactly generated triangulated categories $\mathcal{T}$ and $\mathcal{S}$, together with exact contravariant functors $Q: \mathcal{T}^{\mathrm{op}} \to \mathcal{S}$ and $Q: \mathcal{S}^{\mathrm{op}} \to \mathcal{T}$, satisfying the following conditions:
  \begin{enumerate}
    \item $Q(X)$ is pure injective for any $X \in \mathcal{T}$ or $\mathcal{S}$;
    \item for any $X \in \mathcal{T}$ or $\mathcal{S}$, there is a natural pure monomorphism $i_X: X \to Q^2(X)$;
    \item for any $X \in \mathcal{T}$ or $\mathcal{S}$, the composite $Q(i_X) \circ i_{Q(X)}$ is the identity;
    \item for any set of objects $\{X_i\}_{i \in I}$ in $\mathcal{T}$ or $\mathcal{S}$, we have  $Q\left(\coprod_{i \in I} X_i\right) \simeq \prod_{i \in I} Q(X_i)$.
  \end{enumerate}}
\end{df}

Let \( \mathcal{T} \) be a tensor triangulated category, \( C \in \mathcal{T}^c \) and \( X \in \mathcal{T} \). The \emph{Brown-Comenetz dual} of \( X \) with respect to \( C \) is the unique object, denoted by $\mathbb{I}_C(X)$,
which represents the cohomological functor
$\operatorname{Hom}_{\mathbb{Z}}(\operatorname{Hom}_{\mathcal{T}}(C,X \otimes-), \mathbb{Q}/\mathbb{Z})$, i.e.,
$$\operatorname{Hom}_{\mathbb{Z}}(\operatorname{Hom}_{\mathcal{T}}(C,X \otimes-), \mathbb{Q}/\mathbb{Z})\cong\operatorname{Hom}_{\mathcal{T}}(-,\mathbb{I}_C(X)).$$

\begin{rem}\label{lem:3.2}{\rm (1) Let $\mathcal{T}$ be a compactly generated tensor triangulated category. By \cite[Proposition 3.13]{BW}, the triple $(\mathcal{T}, \mathcal{T}, \mathbb{I})$ is a duality triple, where $\mathbb{I}=\mathbb{I}_{{\i}}({\i})$.

(2) For $C\in\mathcal{T}^c$, one has the following isomorphism
\begin{center}$\operatorname{Hom}_{\mathcal{T}}(-,\mathbb{I}_{{\i}}(C))
\cong\operatorname{Hom}_{\mathbb{Z}}(\operatorname{Hom}_{\mathcal{T}}(C^\vee,-), \mathbb{Q}/\mathbb{Z})\cong\operatorname{Hom}_{\mathcal{T}}(-,(C^\vee)^*)$,\end{center}
it implies that $\mathbb{I}_{{\i}}(C)\cong(C^\vee)^*$. In  particular, $\mathbb{I}\cong{\i}^*$. Hence $\mathbb{I}_{{\i}}(X)\cong\mathcal{H}\mathrm{om}(X,\mathbb{I})
\cong\mathcal{H}\mathrm{om}(X,{\i}^*)$ for any $X\in\mathcal{T}$
by \cite[Lemma 3.12]{BW}.

(3) Let $M$ be a compact silting generator. If $M^\vee\in\mathrm{add}(M)$, then $\mathbb{I}_{{\i}}(M)\in\mathcal{I}$, so $\mathbb{I}_{{\i}}(P)\in\mathcal{I}$ for $P\in\mathcal{P}$. In particular,
if $\mathrm{add}(M)=\mathrm{add}({\i})$, then $M^\vee\in\mathrm{add}(M)$ and so $\mathbb{I}_{{\i}}(\mathcal{P})\subseteq\mathcal{I}$. This fact will be used frequently.}
\end{rem}

\begin{exa}\label{lem:2.2} {\rm
(1) Let $R$ be a ring. The derived category $\mathrm{D}(R)$ is
 a compactly generated tensor triangulated category, whose tensor unit
$R$ is also a compact tilting generator, and $\mathbb{I}_{R}(X)\cong\mathrm{Hom}_R(X,\mathbb{Q}/\mathbb{Z})$ for $X\in\mathrm{D}(R)$.

(2) Let $A$ be a non-positive dg ring. The derived category $\mathrm{D}(A)$ is a compactly generated tensor triangulated category,  in which
$A$  serves both as the tensor unit and as a compact silting generator, and $\mathbb{I}_{A}(X)\cong\mathrm{Hom}_A(X,\mathbb{Q}/\mathbb{Z})$ for $X\in\mathrm{D}(A)$.

(3) Let $X = \mathbb{P}^1_k$ be the projective line over a field $k$ and $A = kQ$ the path algebra of the Kronecker quiver $1 \rightrightarrows 2$.  It is well known that
\[
\mathrm{D}^b(\operatorname{Coh} \mathbb{P}^1_k) \simeq \mathrm{D}^b(\operatorname{mod} A)
\]
 as triangulated categories. Consequently,  $\mathrm{D}^b(\operatorname{Coh} \mathbb{P}^1_k)$ is compactly generated and $M = \mathcal{O}(-1) \oplus \mathcal{O}$ is a compact tilting generator.  However, the tensor unit for the standard tensor structure on $\mathrm{D}^b(\operatorname{Coh}\mathbb{P}^1_k)$ is $\mathcal{O}$.}
\end{exa}

Recall that a morphism $f$ in $\mathcal{T}$ is called \emph{phantom} if $\mathrm{Hom}_\mathcal{T}(C,f)=0$ for all $C\in\mathcal{T}^c$.

\begin{lem}\label{lem1.22} {\rm (\cite[Theorem 3.6]{BW}, \cite[Lemma 1.3]{K})} Let $\mathcal{T}$ be a tensor triangulated category and $X\rightarrow Y\rightarrow Z\stackrel{\Phi}\rightarrow X[1]$ an exact triangle. The triangle is pure exact if and only if $\Phi$ is phantom if and only if $\mathbb{I}_{{\i}}(\Phi)=0$.
\end{lem}

\begin{lem}\label{lem1.2} Let $\mathcal{T}$ be a tensor triangulated category with a compact silting generator $M$ and $X\rightarrow Y\rightarrow Z\stackrel{g}\rightarrow X[1]$ an exact triangle.
 If $M^\vee\in\mathrm{add}(M)$, then $g$ is an $M$-ghost if and only if $\mathbb{I}_{{\i}}(g)$ is an $M$-ghost.
\end{lem}
\begin{proof} For $C\in\mathcal{T}^c$ and $i\in\mathbb{Z}$, it follows from the isomorphism $\mathrm{Hom}_\mathcal{T}(C^\vee,-)\cong\mathrm{Hom}_\mathcal{T}({\i},C\otimes -)$ and the definition of $\mathbb{I}_{{\i}}$, one has
$$\operatorname{Hom}_{\mathcal{T}}(C[i],\mathbb{I}_{{\i}}(-))\cong
\operatorname{Hom}_{\mathcal{T}}(C^\vee,-[i])^+,$$
it yields the following commutative diagram:
\begin{center} $\xymatrix@C=20pt@R=20pt{
 0\ar[r]& \operatorname{Hom}_{\mathcal{T}}(C[i],\mathbb{I}_{{\i}}(Z)) \ar[d]^\cong\ar[r]& \operatorname{Hom}_{\mathcal{T}}(C[i],\mathbb{I}_{{\i}}(Y)) \ar[d]^\cong\ar[r] &\operatorname{Hom}_{\mathcal{T}}(C[i],\mathbb{I}_{{\i}}(X)) \ar[d]^\cong\ar[r]&0 \\
 0\ar[r]& \operatorname{Hom}_{\mathcal{T}}(C^\vee,Z[i])^+  \ar[r]& \operatorname{Hom}_{\mathcal{T}}(C^\vee,Y[i])^+ \ar[r] &\operatorname{Hom}_{\mathcal{T}}(C^\vee,X[i])^+ \ar[r]&0}$
\end{center}
If $g$ is an $M$-ghost, then the lower is exact for $C=M$ as  $M^\vee\in\mathrm{add}(M)$,  so is the upper row, it implies that $\mathbb{I}_{{\i}}(g)$ is an $M$-ghost.
Conversely, the upper row is exact for $C=M^\vee$. As $M^{\vee\vee}\cong M$, it implies that $0\rightarrow\operatorname{Hom}_{\mathcal{T}}(M[-i],X)\rightarrow
\operatorname{Hom}_{\mathcal{T}}(M[-i],Y)\rightarrow \operatorname{Hom}_{\mathcal{T}}(M[-i],Z)\rightarrow0$ is exact. Thus $g$ is an $M$-ghost.
\end{proof}

\begin{lem}\label{lem1.4} Let $\mathcal{T}$ be a tensor triangulated category with a compact silting generator $M$.
If  $M^\vee\in\mathrm{add}(M)$, then for $X\in\mathcal{T}$ and $t\in\mathbb{Z}$, $$\mathcal{H}^{-t}_M(\mathbb{I}_{{\i}}(X))\cong\mathbb{I}_{{\i}}(\mathcal{H}^t_M(X)).$$ In particular,
$\mathrm{sup}\mathbb{I}_{{\i}}(X)=-\mathrm{inf}X$ and $\mathrm{inf}\mathbb{I}_{{\i}}(X)=-\mathrm{sup}X$.
\end{lem}
\begin{proof} For $t\in\mathbb{Z}$,  it follows from Lemma \ref{lem0.2} and the definition of $\mathbb{I}_{{\i}}$ that
\begin{center}$\begin{aligned}\mathrm{Hom}_\mathcal{T}(M,\mathbb{I}_{{\i}}(\mathcal{H}^t_M(X)))
&\cong
\mathrm{Hom}_\mathcal{T}({\i},M\otimes\mathcal{H}^t_M(X))^+\\
&\cong\mathrm{Hom}_\mathcal{T}(M^\vee,\mathcal{H}^{t}_M(X))^+\\
&\cong\mathrm{Hom}_\mathcal{T}(M^\vee,X[t])^+\\
&\cong\mathrm{Hom}_\mathcal{T}({\i},M[t]\otimes X)^+\\
&\cong\mathrm{Hom}_\mathcal{T}(M[t],\mathbb{I}_{{\i}}(X))\\
&\cong\mathrm{Hom}_\mathcal{T}(M,\mathcal{H}^{-t}_M(\mathbb{I}_{{\i}}(X))),\end{aligned}$\end{center}
Hence \cite[Theorem 1.3]{HKM} implies the desired isomorphism.
\end{proof}

\begin{lem}\label{lem1.14} Let $\mathcal{T}$ be a tensor triangulated category with a compact silting generator $M$.
 If $\mathrm{add}(M)=\mathrm{add}({\i})$, then for $X,Y\in\mathcal{T}$, one has $$\mathrm{inf}\mathcal{H}\mathrm{om}(X,Y)\geq\mathrm{inf}Y-\mathrm{sup}X,\ \mathrm{sup}(X\otimes Y)\leq\mathrm{sup}X+\mathrm{sup}Y.$$
\end{lem}
\begin{proof} As $M\in\mathrm{add}({\i})$, $\mathrm{sup}(M\otimes X)\leq\mathrm{sup}X$, it implies that
$$\operatorname{Hom}_{\mathcal{T}}(M[i],\mathcal{H}\mathrm{om}(X,Y))\cong
\operatorname{Hom}_{\mathcal{T}}((M\otimes X)[i],Y)=0$$ for $i>\mathrm{sup}X-\mathrm{inf}Y$. Hence $\mathrm{inf}\mathcal{H}\mathrm{om}(X,Y)\geq\mathrm{inf}Y-\mathrm{sup}X$ by Lemma \ref{lem0.2}. As $\mathbb{I}_{{\i}}(X\otimes Y)\cong\mathcal{H}\mathrm{om}(X,\mathbb{I}_{{\i}}(Y))$, it follows from Lemma \ref{lem1.4} that
$\mathrm{sup}(X\otimes Y)=-\mathrm{inf}\mathbb{I}_{{\i}}(X\otimes Y)\leq-\mathrm{inf}\mathbb{I}_{{\i}}(Y)+\mathrm{sup}X=\mathrm{sup}X+\mathrm{sup}Y$.
\end{proof}

In what follows, $\mathcal{T}$ is a compactly generated tensor triangulated category with compact tensor unit ${\i}$ and $M$ is a compact silting generator with heart $\mathcal{H}_M$.

\bigskip
\section {\bf Flat dimension in tensor triangulated categories}
This section introduces a triangulated analogue of flat dimension $\mathrm{flatdim}_\mathcal{T}X$, and provides several criteria for
$\mathrm{flatdim}_\mathcal{T}X$ in terms of vanishing conditions and resolutions.

Let \(\mathcal{F}=\operatorname{PGen}(M)\) be the class of objects $F\in\mathcal{T}$ for which there is a pure exact triangle $K\rightarrow P\stackrel{f}\rightarrow F\stackrel{h}\rightarrow K[1]$ with $P\in\mathcal{P}$. Equivalently,
$\mathcal{F}$ is a pure quotient of an object of $\mathcal{P}$. We begin by studying the class
 \(\mathcal{F}\),  which serves as the triangulated analogue of flat objects.

\begin{lem}\label{lem1.03} Assume that $M^\vee\in\mathrm{add}(M)$.

$(1)$ An object $F\in\mathcal{F}$ if and only if $\mathbb{I}_{{\i}}(F)\in\mathcal{I}$. Moreover, $\mathcal{F}$ is closed under pure subobjects, pure quotients and pure extensions.

$(2)$ An object $F\in\mathcal{F}$ if and only if every exact triangle  $K\rightarrow P\stackrel{f}\rightarrow F\stackrel{h}\rightarrow K[1]$ with $f$ a sppj morphism is pure.
\end{lem}
\begin{proof} (1) By Remark \ref{lem:3.2}(3), $\mathbb{I}_{{\i}}(F)\in\mathcal{I}$. Conversely, let $K\rightarrow P\stackrel{f}\rightarrow F\stackrel{h}\rightarrow K[1]$ be an exact triangle with $f$ a sppj morphism. Then $\mathbb{I}_{{\i}}(K)[-1]\rightarrow \mathbb{I}_{{\i}}(F)\rightarrow \mathbb{I}_{{\i}}(P)\rightarrow \mathbb{I}_{{\i}}(K)$ is an exact triangle. As  $\mathrm{inf}\mathbb{I}_{{\i}}(K)=-\mathrm{sup}K\geq-\mathrm{sup}F\geq0$ by Lemma  \ref{lem1.4}, $\operatorname{Hom}_{\mathcal{T}}(\mathbb{I}_{{\i}}(K)[-1],\mathbb{I}_{{\i}}(F))\cong
\operatorname{Hom}_{\mathcal{T}}(\mathcal{H}^{-1}_M(\mathbb{I}_{{\i}}(K)),\mathbb{I}_{{\i}}(F))=0$ by Lemma  \ref{lem0.2}, it implies that $h$ is phantom by Lemma  \ref{lem1.22}. Thus $F\in\mathcal{F}$. The last statement follows from \cite[Theorem 4.3]{BW}.

 (2) Let $F\in\mathcal{F}$. As $f$ is a sppj morphism, $K[1]\in\mathcal{T}^{\leq-1}$, it follows from Lemma  \ref{lem1.4} that $\mathbb{I}_{{\i}}(K[1])\in\mathcal{T}^{\geq1}$, so $\operatorname{Hom}_{\mathcal{T}}(\mathbb{I}_{{\i}}(K),\mathbb{I}_{{\i}}(F[-1]))\cong
 \operatorname{Hom}_{\mathcal{T}}(\mathbb{I}_{{\i}}(K[1]),\mathbb{I}_{{\i}}(F))=0$ since $\mathbb{I}_{{\i}}(F)\in\mathcal{I}$ by (1). Hence the triangle $\mathbb{I}_{{\i}}(F)\rightarrow \mathbb{I}_{{\i}}(P)\rightarrow \mathbb{I}_{{\i}}(K)\rightarrow \mathbb{I}_{{\i}}(F[-1])$ is split, and then $h$ is phantom by Lemma  \ref{lem1.22}. The converse is trivial.
\end{proof}

\begin{lem}\label{lem2.20} Let $n\in\mathbb{Z}$, $F\in\mathcal{F}$ and $X\in\mathcal{T}$. If $\mathrm{add}(M)=\mathrm{add}({\i})$, then

 $(1)$ $\mathcal{H}^n_M(F\otimes X)\stackrel{\cong}\rightarrow \mathcal{H}^n_M(F\otimes \tau^{\geq n}(X)),\
\mathcal{H}^n_M(F\otimes \tau^{\leq n}(X))\stackrel{\cong}\rightarrow\mathcal{H}^n_M(F\otimes X)$.

$(2)$ $\mathcal{H}^0_M(F\otimes H)\stackrel{\cong}\rightarrow \mathcal{H}^0_M(F)\otimes H$ for any $H\in\mathcal{H}_M$.

$(3)$ $\mathcal{H}^n_M(F\otimes X)\stackrel{\cong}\rightarrow \mathcal{H}^0_M(F)\otimes\mathcal{H}^n_M(X)$.
\end{lem}
\begin{proof} (1) Consider the exact triangle $\tau^{\leq n-1}(X)\rightarrow X\rightarrow\tau^{\geq n}(X)\rightarrow\tau^{\leq n-1}(X)[1]$, it induces an exact triangle $F\otimes\tau^{\leq n-1}(X)\rightarrow F\otimes X\rightarrow F\otimes\tau^{\geq n}(X)\rightarrow F\otimes\tau^{\leq n-1}(X)[1]$. As $\mathrm{add}(M)=\mathrm{add}({\i})$,
$F\otimes\tau^{\leq n-1}(X)\in\mathcal{T}^{\leq n-1}$  by Lemma \ref{lem1.14}, it implies that $\mathcal{H}^n_M(F\otimes X)\stackrel{\cong}\rightarrow \mathcal{H}^n_M(F\otimes \tau^{\geq n}(X))$.
 The proof of another isomorphism is similar.

(2) By (1) $\mathcal{H}^0_M(F\otimes H)\stackrel{\cong}\rightarrow \mathcal{H}^0_M(\mathcal{H}^0_M(F)\otimes H)$. As  $F\in\mathcal{F}$, there is a pure exact triangle $K\rightarrow P\rightarrow F\rightarrow K[1]$ with  $P\in\mathcal{P}$, so $0\rightarrow\mathcal{H}^0_M(K)\rightarrow \mathcal{H}^0_M(P)\rightarrow \mathcal{H}^0_M(F)\rightarrow 0$ is pure exact in $\mathcal{H}_M$ by \cite[Lemma 2.8]{K1}. Hence $0\rightarrow\mathcal{H}^0_M(K)\otimes H\rightarrow \mathcal{H}^0_M(P)\otimes H\rightarrow \mathcal{H}^0_M(F)\otimes H\rightarrow 0$ is exact in $\mathcal{H}_M$, and then $\mathcal{H}^0_M(\mathcal{H}^0_M(F)\otimes H)\cong\mathcal{H}^0_M(F)\otimes H$, as required.

(3) By (1) and (2), $\mathcal{H}^n_M(F\otimes X)\cong\mathcal{H}^n_M(F\otimes \tau^{\geq n}(X))\cong\mathcal{H}^n_M(F\otimes \tau^{\leq n}(\tau^{\geq n}(X)))\cong\mathcal{H}^n_M(F\otimes \mathcal{H}^n_M(X)[-n])\cong\mathcal{H}^0_M(F\otimes \mathcal{H}^n_M(X))\cong\mathcal{H}^0_M(F)\otimes\mathcal{H}^n_M(X)$.
\end{proof}

\begin{cor}\label{lem1.32} If $\mathrm{add}(M)=\mathrm{add}({\i})$, then $F\in\mathcal{F}$ if and only if
for any $M$-ghost triangle $X\rightarrow Y\rightarrow Z\rightarrow X[1]$, the triangle $F\otimes X\rightarrow F\otimes Y\rightarrow F\otimes Z\rightarrow F\otimes X[1]$ is $M$-ghost.
\end{cor}
\begin{proof} Let $F\in\mathcal{F}$. There is a pure exact triangle $K\rightarrow P\rightarrow F\rightarrow K[1]$ with  $P\in\mathcal{P}$, so $0\rightarrow\mathcal{H}^0_M(K)\rightarrow \mathcal{H}^0_M(P)\rightarrow \mathcal{H}^0_M(F)\rightarrow 0$ is pure exact in $\mathcal{H}_M$ by \cite[Lemma 2.8]{K1}. Hence $0\rightarrow\mathcal{H}^0_M(F)\otimes\mathcal{H}^n_M(X)\rightarrow
\mathcal{H}^0_M(F)\otimes\mathcal{H}^n_M(Y)\rightarrow\mathcal{H}^0_M(F)\otimes\mathcal{H}^n_M(Z)
\rightarrow0$ is exact for all $n\in\mathbb{Z}$, it follows from Lemmas \ref{lem2.20} and \ref{lem2.0} that $F\otimes X\rightarrow F\otimes Y\rightarrow F\otimes Z\rightarrow F\otimes X[1]$ is $M$-ghost. Conversely, $\mathcal{H}\mathrm{om}(Z,\mathbb{I}_{{\i}}(F))\rightarrow \mathcal{H}\mathrm{om}(Y,\mathbb{I}_{{\i}}(F))\rightarrow \mathcal{H}\mathrm{om}(X,\mathbb{I}_{{\i}}(F))\rightarrow \mathcal{H}\mathrm{om}(Z,\mathbb{I}_{{\i}}(F))[1]$ is $M$-ghost by Lemma \ref{lem1.2}. For $i\in\mathbb{Z}$, one has the following commutative diagram:
\begin{center} $\xymatrix@C=15pt@R=20pt{
 0\ar[r]& \operatorname{Hom}_{\mathcal{T}}({\i}[i],\mathbb{I}_{{\i}}(F\otimes Z))\ar[d]^\cong\ar[r]& \operatorname{Hom}_{\mathcal{T}}({\i}[i],\mathbb{I}_{{\i}}(F\otimes Y))\ar[d]^\cong\ar[r] &\operatorname{Hom}_{\mathcal{T}}({\i}[i],\mathbb{I}_{{\i}}(F\otimes X))\ar[d]^\cong\ar[r]&0 \\
0\ar[r]& \operatorname{Hom}_{\mathcal{T}}(Z[i],\mathbb{I}_{{\i}}(F))\ar[r]& \operatorname{Hom}_{\mathcal{T}}(Y[i],\mathbb{I}_{{\i}}(F))\ar[r] &\operatorname{Hom}_{\mathcal{T}}(X[i],\mathbb{I}_{{\i}}(F))\ar[r]&0 }$
\end{center}As ${\i}\in\mathrm{add}(M)$, the upper row is exact by Lemma \ref{lem1.2}, so is the lower row. Thus
  $\mathbb{I}_{{\i}}(F)\in\mathcal{I}$, and then $F\in\mathcal{F}$ by Lemma \ref{lem1.03}.
\end{proof}

Following the classical characterization of flat dimension for complexes, we define the flat dimension of objects as follows.

\begin{df}\label{lem:2.12} {\rm The \emph{flat dimension} $\mathrm{flatdim}_\mathcal{T}X$ of $X\in\mathcal{T}$ is defined by
\begin{center}$\mathrm{flatdim}_\mathcal{T}X=\mathrm{inf}\{n\in\mathbb{Z}\hspace{0.03cm}|\hspace{0.03cm}
\mathcal{H}^{-i}_M(U\otimes X)=0\ \textrm{for\ any}\ U\in\mathcal{T}^+\ \textrm{and}\ i>n-\mathrm{inf}U\}$.\end{center}}
\end{df}

The next result provides several vanishing criteria for flat dimension.

\begin{thm}\label{lem2.1} Let $X\in\mathcal{T}$ and $n\in\mathbb{Z}$. Then the following are equivalent:

$(1)$ $\mathrm{flatdim}_\mathcal{T}X\leq n$;

$(2)$ $\mathcal{H}^{-i}_M(U\otimes X)=0$ for any $U\in\mathcal{T}^b$ and $i>n-\mathrm{inf}U$;

$(3)$ $\mathcal{H}^{-i}_M(H\otimes X)=0$ for any $H\in\mathcal{H}_M$ and $i>n$.\\
In particular, $\mathrm{flatdim}_\mathcal{T}X=\mathrm{sup}\{n\in\mathbb{Z}\hspace{0.03cm}|\hspace{0.03cm}
\mathcal{H}^{-n}_M(H\otimes X)\neq0\ \textrm{for}\ H\in\mathcal{H}_M\}$.
\end{thm}
\begin{proof} (1) $\Rightarrow$ (3) is trivial. (3) $\Rightarrow$ (2) follows by the induction on $\mathrm{amp}U$.

(2) $\Rightarrow$ (1) Let $U\in\mathcal{T}^{+}$. There exists a sequence in  $\mathcal{T}$
$$\tau^{\leq\mathrm{inf}U}(U)\stackrel{\nu_0}\longrightarrow
\tau^{\leq\mathrm{inf}U+1}(U)\stackrel{\nu_1}\longrightarrow
\tau^{\leq\mathrm{inf}U+2}(U)\stackrel{\nu_2}\longrightarrow\cdots,$$ where $\tau^{\leq\mathrm{inf}U+j}(U)\in\mathcal{T}^\mathrm{b}$ and $\mathrm{inf}\tau^{\leq\mathrm{inf}U+j}(U)=\mathrm{inf}U$ for $j\geq 0$, it yields an exact triangle $$\bigoplus_{j\geq0}\tau^{\leq\mathrm{inf}U+j}(U)\longrightarrow
\bigoplus_{j\geq0}\tau^{\leq\mathrm{inf}U+j}(U)\longrightarrow U\longrightarrow\bigoplus_{j\geq0}\tau^{\leq\mathrm{inf}U+j}(U)[1].$$ Hence  the following exact sequence $$\bigoplus\mathcal{H}^{-i}_M(\tau^{\leq\mathrm{inf}U+j}(U)\otimes X)\rightarrow\mathcal{H}^{-i}_M(U\otimes X)
\rightarrow\bigoplus\mathcal{H}^{-i+1}_M(\tau^{\leq\mathrm{inf}U+j}(U)\otimes X)$$
 implies that $\mathcal{H}^{-i}_M(U\otimes X)=0$ for $i>n+1-\mathrm{inf}U$. Consider the exact triangle $$\tau^{\geq\mathrm{inf}U+1}(U)[-1]\rightarrow\mathcal{H}^{\mathrm{inf}U}_M(U)[-\mathrm{inf}U]
 \rightarrow U\rightarrow\tau^{\geq\mathrm{inf}U+1}(U).$$ For $i'=n+1-\mathrm{inf}U$, it follows from the exact sequence $$0=\bigoplus\mathcal{H}^{-i'}_M(\mathcal{H}^{\mathrm{inf}U}_M(U)[-\mathrm{inf}U]\otimes X)\rightarrow\mathcal{H}^{-i'}_M(U\otimes X)
\rightarrow\bigoplus\mathcal{H}^{-i'}_M(\tau^{\geq\mathrm{inf}U+1}(U)\otimes X)=0$$ that $\mathcal{H}^{-i}_M(U\otimes X)=0$ for $i>n-\mathrm{inf}U$.
\end{proof}

Let $X\in\mathcal{T}^{-}$ and set $s=\mathrm{sup}X$. As $\mathcal{P}\subseteq\mathcal{F}$, there is $f:F\rightarrow X$ with $F\in\mathcal{F}[-s]$, such that $\mathcal{H}^s_M(f)$ is surjective. We call $f:F\rightarrow X$ a \emph{spft morphism}.  A \emph{spft resolution} $F_\bullet$ of $X=X_0$ is a sequence of exact triangles
\begin{align}
X_{n+1}\stackrel{g_{n+1}}\longrightarrow F_n\stackrel{f_n}\longrightarrow X_n\longrightarrow X_{n+1}[1]
\label{exact03}
\tag{$\dag_n$}\end{align}
such that $f_n$ is a spft morphism for $n\geq 0$. Set $d_n=g_{n}f_n:F_n\rightarrow F_{n-1}$, and write $$\cdots\longrightarrow F_n\stackrel{d_n}\longrightarrow F_{n-1}\longrightarrow\cdots\longrightarrow F_1\stackrel{d_1}\longrightarrow F_0\longrightarrow X.$$

\begin{lem}\label{lem2.20'} Assume that $\mathrm{add}(M)=\mathrm{add}({\i})$, and let $X\in\mathcal{T}^-$. Then $\mathrm{flatdim}_\mathcal{T}X\geq-\mathrm{sup}X$ and $X\in\mathcal{F}[-\mathrm{sup}X]$ if and only if $\mathrm{flatdim}_\mathcal{T}X=-\mathrm{sup}X$.
\end{lem}
\begin{proof} For $i<\mathrm{inf}\mathbb{I}_{{\i}}(X)=-\mathrm{sup}X$, as $\mathrm{Hom}_\mathcal{T}(\mathcal{H}_M[-i],\mathbb{I}_{{\i}}(X))=0$, $\mathrm{Hom}_\mathcal{T}({\i},\mathcal{H}_M[-i]\otimes X)^+=0$. Since
 $M\in\mathrm{add}({\i})$, it implies that $\mathrm{Hom}_\mathcal{T}(M[i],\mathcal{H}_M\otimes X)=0$ and so $\mathcal{H}^{-i}_M(\mathcal{H}_M\otimes X)=0$ by Lemma \ref{lem0.2}. Thus $\mathrm{flatdim}_\mathcal{T}X\geq-\mathrm{sup}X$ by Theorem \ref{lem2.1}.
Let $X\in\mathcal{F}[-\mathrm{sup}X]$. As $\mathrm{flatdim}_\mathcal{T}X[\mathrm{sup}X]=\mathrm{flatdim}_\mathcal{T}X+\mathrm{sup}X$, it follows from Lemma \ref{lem2.20} that $\mathrm{flatdim}_\mathcal{T}X=-\mathrm{sup}X$. Conversely, consider the exact triangle $Y\rightarrow F\stackrel{f}\rightarrow X[\mathrm{sup}X]\stackrel{g}\rightarrow Y[1]$ with $f$ a spft morphism. Then $\mathrm{sup}Y\leq0$ and the triangle $\mathbb{I}_{{\i}}(Y)[-1]\rightarrow \mathbb{I}_{{\i}}(X[\mathrm{sup}X])\rightarrow\mathbb{I}_{{\i}}(F)\rightarrow \mathbb{I}_{{\i}}(Y)$ is exact.  As $\mathrm{flatdim}_\mathcal{T}X=-\mathrm{sup}X$, $\mathrm{Hom}_\mathcal{T}(M[i],\mathbb{I}_{{\i}}(Y)\otimes X)=0$ for $i>-\mathrm{sup}X$, it implies that
$\mathrm{Hom}_\mathcal{T}(\mathbb{I}_{{\i}}(Y)[-1], \mathbb{I}_{{\i}}(X[\mathrm{sup}X]))\cong
\mathrm{Hom}_\mathcal{T}({\i}[1-\mathrm{sup}X],\mathbb{I}_{{\i}}(Y)\otimes X)^+=0$ as ${\i}\in\mathrm{add}(M)$. Hence $g$ is phantom  and $X[\mathrm{sup}X]\in\mathcal{F}$ by Lemma \ref{lem1.03}.
\end{proof}

\begin{lem}\label{lem3.4} Assume $\mathrm{add}(M)=\mathrm{add}({\i})$. Let $X\in\mathcal{T}^{-}$ and $X'\rightarrow F\stackrel{f}\rightarrow X\rightarrow X'[1]$ be an exact triangle with $f$ a spft morphism. If $\mathrm{flatdim}_\mathcal{T}X+\mathrm{sup}X\geq1$, then $\mathrm{flatdim}_\mathcal{T}X'=\mathrm{flatdim}_\mathcal{T}X-1$.
\end{lem}
\begin{proof} Set $s=\mathrm{sup}X$ and $\mathrm{flatdim}_\mathcal{T}X+s=d\geq1$, then $\mathcal{H}^{-i}_M(\mathcal{H}_M\otimes X')\cong\mathcal{H}^{-i-1}_M(\mathcal{H}_M\otimes X)=0$ for $i>-s+d-1$ by Lemma \ref{lem2.20} and Theorem \ref{lem2.1}, it implies that $\mathrm{flatdim}_\mathcal{T}X'\leq d-s-1$. Also there is $H\in\mathcal{H}_M$, such that $\mathcal{H}^{-d+s}_M(H\otimes X)\neq0$. Then the exact sequence
$$0=\mathcal{H}^{-d+s}_M(H\otimes F)\rightarrow\mathcal{H}^{-d+s}_M(H\otimes X)\rightarrow\mathcal{H}^{-d+s+1}_M(H\otimes X')$$ yields that
$\mathcal{H}^{-d+s+1}_M(H\otimes X')\neq0$, so $\mathrm{flatdim}_\mathcal{T}X'=d-s-1$.
\end{proof}

The next theorem gives a resolution-theoretic characterization of $\mathrm{flatdim}_\mathcal{T}X$ for $X\in\mathcal{T}^-$, analogous to the classical complex case and essential in Section 3.

\begin{thm}\label{lem2.21} Let $X\in\mathcal{T}^-$ and $d\geq0$. If $\mathrm{add}(M)=\mathrm{add}({\i})$, then the next are equivalent:

$(1)$ $\mathrm{flatdim}_\mathcal{T}X+\mathrm{sup}X=d$;

$(2)$  For any spft resolution $F_\bullet$ of $X$, there exists $e\geq0$ such that $X_e\in\mathcal{F}[-\mathrm{sup}X_e]$, $e+\mathrm{sup}X-\mathrm{sup}X_e=d$ and
$g_e:X_e\rightarrow F_{e-1}$
 is not a pure monomorphism;

$(3)$  $X$ has a spft resolution $X_{e}\rightarrow F_{e-1}\rightarrow\cdots\rightarrow F_0\rightarrow X$ such that $X_e\in\mathcal{F}[-\mathrm{sup}X_e]$, $e+\mathrm{sup}X-\mathrm{sup}X_e=d$ and
$g_e:X_e\rightarrow F_{e-1}$ is not a pure monomorphism;

$(4)$ $d$ is the smallest number which satisfies $X\in\mathcal{F}[-\mathrm{sup}X]\ast\cdots\ast\mathcal{F}[-\mathrm{sup}X+d]$.
\end{thm}
\begin{proof} (1) $\Rightarrow$ (2) Set $s=\mathrm{sup}X$. We proceed by induction on $d$. The case of $d=0$ is by Lemma  \ref{lem2.20'}. Assume $d\geq1$. For any spft resolution $F_\bullet$ of $X$, consider the exact triangle $$X_1\rightarrow F_0\rightarrow X\rightarrow X_1[1].$$ Then $\mathrm{flatdim}_\mathcal{T}X_1=d-s-1=-\mathrm{sup}X_1+(d-s-1+\mathrm{sup}F_1)$ by Lemma \ref{lem3.4}.
 As $\cdots\rightarrow F_n\rightarrow\cdots\rightarrow F_1\rightarrow X_1$ is a spft resolution of $X_1$ and $\mathrm{sup}F_1\leq\mathrm{sup}F_0$, it follows by the induction that there is $e\geq0$ such that $X_e\in\mathcal{F}[-\mathrm{sup}X_e]$, $e-1+\mathrm{sup}X_1-\mathrm{sup}F_e=d-1-s+\mathrm{sup}F_1$ and $g_e:X_e\rightarrow F_{e-1}$
 is not a pure monomorphism.

(2) $\Rightarrow$ (3) is trivial.

(3) $\Rightarrow$ (4) We proceed by induction on $d$. If $d=0$, then $X\in\mathcal{F}[-\mathrm{sup}X]$ by Lemma \ref{lem2.20'}.
  Assume that $d\geq1$ and $X'\stackrel{b}\rightarrow F'\stackrel{a}\rightarrow X\rightarrow X'[1]$ is an exact triangle with $a$ a spft morphism. Then the following triangle $$X'\oplus F'\stackrel{\left[\begin{smallmatrix}b &0\\ 0 &1\end{smallmatrix}\right]}\longrightarrow F'\oplus F'\stackrel{[a,0]}\longrightarrow X\longrightarrow (X'\oplus F')[1]$$is exact. Set $F=F'\oplus F'$ and  $X_1=X'\oplus F'$. Then $\mathrm{sup}X_1=s$ and $f=[a,0]$ is a spft morphism, and so $\mathrm{flatdim}_\mathcal{T}X_1+s=d-1$ by Lemma \ref{lem3.4}. Thus $X_1\in\mathcal{F}[-s]\ast\cdots\ast\mathcal{F}[-s+d-1]$ by the induction, so $X\in\mathcal{F}[-s]\ast\cdots\ast\mathcal{F}[-s+d]$. If $X\in\mathcal{F}[-s]\ast\cdots\ast\mathcal{F}[-s+d-1]$, then  $\mathcal{H}^{-d+s}_M(H\otimes X)=0$ by Lemma \ref{lem2.20},  a contradiction.

(4) $\Rightarrow$ (1) We proceed by induction on $d$. The case of $d=0$ is by Lemma \ref{lem2.20'}. Assume $d\geq1$. There is an exact triangle $Y\rightarrow F\stackrel{f}\rightarrow X\rightarrow Y[1]$, such that $d-1$ is the smallest number that satisfies $Y\in\mathcal{F}[-s]\ast\cdots\ast\mathcal{F}[-s+d-1]$ and $F\in\mathcal{F}[-s]$. Then $\mathrm{flatdim}_\mathcal{T}Y+s=d-1$ by the induction. Note that $f$ is a spft morphism, it follows from Lemma \ref{lem3.4} that $\mathrm{flatdim}_\mathcal{T}X+s=\mathrm{flatdim}_\mathcal{T}Y+s+1=d$.
\end{proof}

\begin{rem}\label{lem:2.0}{\rm (1) For a ring $R$, $\mathcal{F}$ is precisely the class of flat $R$-module, and $\mathrm{flatdim}_RX$ is the classical flat dimension of a complex $X\in\mathrm{D}(R)$.

 (2) For a non-positive dg ring $A$, $\mathcal{F}$ is the class defined in \cite[Definition 4.5]{Mi18} by Lemma \ref{lem2.20'}. So $\mathrm{flatdim}_AX$ is the flat dimension of a dg module $X\in\mathrm{D}(A)$, see \cite[Definition 2.1]{BSSW}.

 (3) If $\mathrm{add}(M)=\mathrm{add}({\i})$, then $\mathcal{F}$ agrees with the class introduced in \cite{Y1} by Lemma  \ref{lem1.03}. It then follows from Theorem \ref{lem2.21} that for $X\in\mathcal{T}^{-}$, the number $\mathrm{flatdim}_\mathcal{T}X+\mathrm{sup}X$ coincides with the flat dimension defined in \cite[Definition 3.10]{Y1}.

(4)  Let $M\in\mathcal{T}^b$ and $\mathrm{add}(M)=\mathrm{add}({\i})$. Following \cite{Y1}, $\mathcal{F}'$ denotes the class of objects in $\mathcal{T}$ satisfying $F=0$, or $\mathrm{sup}F=0$ and $\mathrm{Hom}_\mathcal{T}(F,C^\ast[i])=0$ for all $C\in\mathcal{T}^c$ and $i>\mathrm{sup}C^\ast$. Let $F\in\mathcal{F}$ and $C\in\mathcal{T}^c$. Then $C^*\cong\mathbb{I}_{{\i}}(C^\vee)$ by Remark \ref{lem:3.2}(2). So  by Lemma  \ref{lem0.2}, $$\mathrm{Hom}_\mathcal{T}(F,C^*[i])\cong
 \mathrm{Hom}_\mathcal{T}(F,\mathbb{I}_{{\i}}(C^\vee)[i])\cong
 \mathrm{Hom}_\mathcal{T}(C^\vee\otimes F,{\i}[i])\cong
 \mathrm{Hom}_\mathcal{T}(\mathcal{H}^{-i}_M(C^\vee\otimes F),{\i})=0$$  for $i>-\mathrm{inf}C^\vee=\mathrm{sup}C^*$, it implies that
 $\mathcal{F}=\mathcal{F}'$ by \cite[Remark 3.9(2)]{Y1}.}
\end{rem}

\bigskip
\section {\bf $M$-ghost dimension of $\mathcal{T}$ and (weak) global dimension of $\mathcal{H}_M$}
In this section, we introduce the $M$-ghost length and $M$-ghost dimension, with the central aim of establishing the relationship between the $M$-ghost dimension of $\mathcal{T}$, the (weak) global dimension of $\mathcal{H}_M$, and the generation of $\mathcal{T}^c$ by
$\mathrm{add}(M)$, respectively.

\begin{df}\label{lem:3.0'}{\rm (1) An object $X$ in $\mathcal{T}$ is said to \emph{have
$M$-ghost length $n$}, $\mathrm{gh.len}_MX=n$, if every composite
$X\xrightarrow{f_1} Y_1\xrightarrow{f_2}\cdots\xrightarrow{f_{n+1}}Y_{n+1}$
of $n+1$ $M$-ghosts is $0$, and there exists a composite of $n$ $M$-ghosts starting from $X$ that is nonzero.

 (2) $\mathrm{gh.dim}\mathcal{T}=\mathrm{sup}\{\mathrm{gh.len}_MX\hspace{0.02cm}|\hspace{0.02cm}X\in\mathcal{T}^c\}$, called the \emph{$M$-ghost dimension} of $\mathcal{T}$.}
\end{df}

\begin{lem}\label{lem5.2'} If $M$ is tilting and $X\in\mathcal{T}$ with each $\mathcal{H}^{\ell}_M(X)\in\mathcal{P}$, then $X\cong\coprod_{\ell\in\mathbb{Z}}\mathcal{H}^{\ell}_M(X)[-\ell]$.
\end{lem}
\begin{proof} As $\mathcal{H}^{\ell}_M(X)\in\mathcal{P}$ for $\ell\in\mathbb{Z}$, it follows from Lemma \ref{lem0.2} that $$\mathrm{Hom}_\mathcal{T}(\coprod_{\ell\in\mathbb{Z}}\mathcal{H}^{\ell}_M(X)[-\ell],X)
\cong\prod_{\ell\in\mathbb{Z}}\mathrm{Hom}_\mathcal{T}(\mathcal{H}^{\ell}_M(X),\mathcal{H}^{\ell}_M(X)).$$
So there is $u:\coprod_{\ell\in\mathbb{Z}}\mathcal{H}^{\ell}_M(X)[-\ell]\rightarrow X$ such that $\mathcal{H}^{\ell}_M(u)=\mathrm{id}_{\mathcal{H}^{\ell}_M(X)}$ for $\ell\in\mathbb{Z}$. Extending $u$ to an exact triangle $Y[-1]\rightarrow\coprod_{\ell\in\mathbb{Z}}\mathcal{H}^{\ell}_M(X)[-\ell]\stackrel{u}\rightarrow X\rightarrow Y$. Since  $\mathcal{H}^{\ell}_M(Y)=0$ for all $\ell\in\mathbb{Z}$, it follows from \cite[Lemma 2.8]{Y1} that $Y=0$, and then $X\cong\coprod_{\ell\in\mathbb{Z}}\mathcal{H}^{\ell}_M(X)[-\ell]$.
\end{proof}

Let $M$ be tilting. Denote $\mathrm{gldim}\mathcal{H}_M=
\mathrm{sup}\{\mathrm{pd}_{\mathcal{H}_M}H\hspace{0.02cm}|\hspace{0.02cm}H\in\mathcal{H}_M\}$, where
 $$\mathrm{pd}_{\mathcal{H}_M}H=\mathrm{inf}\{n\geq0\hspace{0.02cm}|\hspace{0.02cm}\exists\ \textrm{an\ exact\ sequence}\ 0\rightarrow P_n\rightarrow\cdots\rightarrow P_0\rightarrow H\rightarrow0\ \textrm{with\ each}\ P_i\in\mathcal{P}\}.$$
 The  next theorem proves that $\mathrm{gldim}\mathcal{H}_M$ is exactly the bound controlling
$M$-ghost length, it is a triangulated analogue of \cite[Theorem 8.14]{EFHO} and \cite[Theorem 1.5]{HL}.

\begin{thm}\label{lem5.2} Let $n\geq0$. If $M$ is tilting, then $\mathrm{gldim}\mathcal{H}_M\leq n$ if and only if $\mathrm{gh.len}_MX\leq n$ for all $X\in\mathcal{T}$. In particular, $\mathrm{gldim}\mathcal{H}_M=
\mathrm{sup}\{\mathrm{gh.len}_MX\hspace{0.02cm}|\hspace{0.02cm}X\in\mathcal{T}\}$.
\end{thm}
\begin{proof} `Only if' part. Let $X\in\mathcal{T}$. Then $\mathrm{pd}_{\mathcal{H}_M}\mathcal{H}^{\ell}_M(X)\leq n$ for all $\ell\in\mathbb{Z}$. We proceed by induction on $n$. If $n=0$, then $\mathcal{H}^{\ell}_M(X)\in\mathcal{P}$ for $\ell\in\mathbb{Z}$, and so $X\cong\coprod_{\ell\in\mathbb{Z}}\mathcal{H}^{\ell}_M(X)[-\ell]\in\langle\mathcal{P}\rangle_1$ by Lemma \ref{lem5.2'}. Assume that $n\geq1$.  By Lemma \ref{lem1.00}, there is an $M$-ghost triangle $X'\rightarrow P\rightarrow X\rightarrow X'[1]$ such that $P\in\langle\mathcal{P}\rangle_{1}$. Then  $0\rightarrow\mathcal{H}^{\ell}_M(X')\rightarrow\mathcal{H}^{\ell}_M(P)\rightarrow \mathcal{H}^{\ell}_M(X)\rightarrow0$ is exact, it implies that $\mathrm{pd}_{\mathcal{H}_M}\mathcal{H}^{\ell}_M(X')\leq n-1$ for all $\ell\in\mathbb{Z}$ as $\mathcal{H}^{\ell}_M(P)\in\mathcal{P}$. By the induction, $X'\in\langle\mathcal{P}\rangle_{n}$, so $X\in\langle\mathcal{P}\rangle_{n+1}$, and hence $\mathrm{gh.len}_MX\leq n$ by \cite[Lemma 4.5]{Y1}.

`If' part. Let $H\in\mathcal{H}_M$ and set $H_0=H$. Fix a projective resolution
$$\cdots\rightarrow P_n\rightarrow P_{n-1}\rightarrow\cdots\rightarrow P_0\rightarrow H_0\rightarrow0$$ in $\mathcal{H}_M$ with the corresponding exact sequence $0\rightarrow H_{s+1}\rightarrow P_s\rightarrow H_s\rightarrow 0$ in $\mathcal{H}_M$. For $s\geq0$, one has a commutative diagram of exact triangles
\begin{center} $\xymatrix@C=30pt@R=20pt{
P_{s-1}[s-2] \ar@{=}[d]\ar[r]& H_{s-1}[s-2]\ar[d]\ar[r]^{\ f_{s}[s-2]} &H_{s}[s-1]\ar[d]\ar[r]&P_{s-1}[s-1]\ar@{=}[d] \\
P_{s-1}[s-2] \ar[r]& T_{s-1}\ar[d]\ar[r] &T_s\ar[d]\ar[r]&P_{s-1}[s-1] \\
     & H_0\ar[d]^{g_{s-1}}\ar@{=}[r]& H_0\ar[d]^{g_s}  \\
     & H_{s-1}[s-1]\ar[r]^{\ \ \ f_{s}[s-1]}& H_{s}[s]  }$
\end{center}where $T_1=P_0$, $g_1=f_1$ and $f_s[s-1]\in\mathcal{G}_M$ for $s\geq1$ by Lemma \ref{lem2.0}. As $\mathrm{gh.len}_MH\leq n$ and $g_{n+1}=f_{n+1}[n]\cdots f_2[1]f_1\in\mathcal{G}^{n+1}_M$, it follows that the triangle $$H_{n+1}[n]\longrightarrow T_{n+1}\longrightarrow H\stackrel{g_{n+1}}\longrightarrow H_{n+1}[n+1]$$ splits, and then $H\in\mathcal{P}\ast\mathcal{P}[1]\ast\cdots\ast\mathcal{P}[n]$, so $\mathrm{Hom}_\mathcal{T}(H_{n},\mathcal{H}_M[1])\cong
\mathrm{Hom}_\mathcal{T}(H,\mathcal{H}_M[n+1])=0$ by \cite[Lemma 4.7]{Y1}. Thus the triangle $H_{n+1}\rightarrow P_{n}\rightarrow H_n\rightarrow H_{n+1}[1]$ splits and $H_n\in\mathcal{P}$, it yields that $\mathrm{pd}_{\mathcal{H}_M}H\leq n$, as required.
\end{proof}

Let $R$ be a ring and $0\not\simeq X\in\mathrm{D}(R)$. The \emph{projective dimension} of $X$ is defined as
\begin{center}$\mathrm{pd}_RX=\mathrm{inf}\{n\in\mathbb{Z}\hspace{0.03cm}|\hspace{0.03cm}
\mathrm{H}^i(\mathrm{RHom}_R(X,N))=0\ \textrm{for\ any}\ R\textrm{-module}\ N\ \textrm{and}\ i>n\}$.\end{center}
By  \cite[Lemma 4.5]{Y1}, $\mathrm{gh.len}_RX$ is precisely the projective level, which was introduced by Avramov, Buchweitz,
Iyengar and Miller in \cite{ABIM}. Upper bounds for projective level are relatively easy to obtain by a projective resolution. We now give a lower bound.

\begin{cor}\label{lem6.6} {\rm (\cite[Theorem 2.1]{AGMSV})} Let $R$ be a ring.
For $0\not\simeq X\in\mathrm{D}^b(R)$, one has
$$\mathrm{pd}_RX+\mathrm{inf}X\leq\mathrm{gh.len}_RX.$$
\end{cor}
\begin{proof} By \cite[Remark 2.2 and Theorem 2.22]{Mi18}, $\mathrm{pd}_RX+\mathrm{sup}X=n$ if and only if $n$ is the smallest number which satisfies $X\in\mathrm{Add}(R)[-\mathrm{sup}X]\ast\cdots\ast\mathrm{Add}(R)[-\mathrm{sup}X+n]$. As $\mathrm{pd}_RX+\mathrm{inf}X\leq\mathrm{pd}_RX+\mathrm{sup}X$, the desired inequality follows from Theorem \ref{lem5.2}.
\end{proof}

\begin{cor}\label{lem5.6} {\rm (\cite[Theorem 1.5]{HL})} Let $R$ be a  ring and denote by $\mathrm{gldim} R$ the global dimension of $R$. Then
 the maximum $R$-ghost length of an object in $\mathrm{D}(R)$ is $\mathrm{gldim}R$.
\end{cor}

\begin{cor}\label{lem0.3} The following are equivalent:

 $(1)$ $\mathcal{T}=\langle\mathcal{P}\rangle_1$;

 $(2)$ $\mathcal{T}=\langle\mathcal{I}\rangle_1$;

 $(3)$ Any $M$-ghost triangle $X\rightarrow Y\rightarrow Z\stackrel{f}\rightarrow X[1]$ splits;

 $(4)$ $M$ is tilting and $\mathrm{gldim}\mathcal{H}_M=0$.
\end{cor}
\begin{proof} (1) $\Rightarrow$ (3) As $Z\in\langle\mathcal{P}\rangle_1$ and $f\in\mathcal{G}_M$, one has $f=0$ and the triangle splits.

(3) $\Rightarrow$ (1) For any $X\in\mathcal{T}$, there is an $M$-ghost triangle $Y[-1]\rightarrow P\rightarrow X\stackrel{f}\rightarrow Y$ with $f\in\mathcal{G}_M$ and $P\in\langle\mathcal{P}\rangle_1$ by Lemma \ref{lem1.00}(1). Then $X$ is a retract of $P$, it implies that $X\in\langle\mathcal{P}\rangle_1$.

(2) $\Leftrightarrow$ (3) The proof is dual to that of (1) $\Leftrightarrow$ (3), by applying Lemma \ref{lem1.00}(2).

(3) $\Rightarrow$ (4) As the triangle $\tau^{\leq-1}(M)\rightarrow M\rightarrow \mathcal{H}^0_M(M)\stackrel{f}\rightarrow\tau^{\leq-1}(M)[1]$ is $M$-ghost, $f=0$, so $\tau^{\leq-1}(M)\in\mathcal{P}$. Then $\mathrm{Hom}_\mathcal{T}(\tau^{\leq-1}(M),M)\cong
\mathrm{Hom}_\mathcal{T}(\tau^{\leq-1}(M),\mathcal{H}^0_M(M))=0$ by Lemma \ref{lem0.2}. Thus $\tau^{\leq-1}(M)=0$ and $M$ is tilting. For $H\in\mathcal{H}_M$, let $K\rightarrow P\stackrel{g}\rightarrow H\stackrel{h}\rightarrow K[1]$ be a triangle with $g$ a sppj morphism. As $P\in\mathcal{H}_M$, $h\in\mathcal{G}_M$ and $H\in\mathcal{P}$. Hence $\mathrm{gldim}\mathcal{H}_M=0$.

(4) $\Rightarrow$ (1) This follows from Theorem \ref{lem5.2} and \cite[Lemma 4.5]{Y1}.
\end{proof}

We now move from the global dimension of the heart to the weak global dimension. Let $M$ be tilting. Denote $\mathrm{wgldim}\mathcal{H}_M=
\mathrm{sup}\{\mathrm{fd}_{\mathcal{H}_M}H\hspace{0.02cm}|\hspace{0.02cm}H\in\mathcal{H}_M\}$, where
 $$\mathrm{fd}_{\mathcal{H}_M}H=\mathrm{inf}\{n\geq0\hspace{0.02cm}|\hspace{0.02cm}\exists\ \textrm{an\ exact\ sequence}\ 0\rightarrow F_n\rightarrow\cdots\rightarrow F_0\rightarrow H\rightarrow0\ \textrm{with\ each}\ F_i\in\mathcal{F}\}.$$

 \begin{lem}\label{lem2.11} If $M$ is tilting and $\mathrm{add}(M)=\mathrm{add}({\i})$, then
$\mathrm{flatdim}_{\mathcal{T}}H=\mathrm{fd}_{\mathcal{H}_M}H$ for $H\in\mathcal{H}_M$.
\end{lem}
\begin{proof} Let $\mathrm{flatdim}_{\mathcal{T}}H=t$.
Consider the spft resolution
$$\cdots\rightarrow P_n\rightarrow P_{n-1}\rightarrow\cdots\rightarrow P_1\rightarrow P_0\rightarrow H$$ with the corresponding exact triangle $H_{j+1}\rightarrow P_j\rightarrow H_j\stackrel{h_{j+1}}\rightarrow H_{j+1}[1]$ for $j\geq0$, where $H_0=H$. Then $\mathcal{H}^{-i}_M(\mathcal{H}_M\otimes H_{t})\cong\mathcal{H}^{-t-i}_M(\mathcal{H}_M\otimes H)=0$ for $i>0$ by Lemma \ref{lem2.20}, and so $\mathrm{Hom}_\mathcal{T}(\mathbb{I}_{{\i}}(H_{t+1})[-1], \mathbb{I}_{{\i}}(H_{t}))\cong
\mathrm{Hom}_\mathcal{T}({\i}[1],\mathbb{I}_{{\i}}(H_{t+1})\otimes H_{t})^+=0$, it implies that $h_{t+1}$ is phantom by Lemma \ref{lem1.22} and
$H_t\in\mathcal{F}$. Note that $h_{i+1}\in\mathcal{G}_M$, so $H_i\in\mathcal{H}_M$ for $i\geq0$.
Hence  $$0\rightarrow H_t\rightarrow P_{t-1}\rightarrow\cdots\rightarrow P_1\rightarrow P_0\rightarrow H\rightarrow0$$ is exact in $\mathcal{H}_M$ by Lemma \ref{lem2.0}, it means that $\mathrm{fd}_{\mathcal{H}_M}H\leq t$. Conversely, let $0\rightarrow Q_r\rightarrow Q_{r-1}\rightarrow\cdots\rightarrow Q_0\rightarrow H\rightarrow0$ be an exact sequence in $\mathcal{H}_M$ with each $Q_{i}\in\mathcal{F}$. Set $K_i=\mathrm{coker}(Q_i\rightarrow Q_{i-1})$. Then $K_{i+1}\rightarrow Q_i\stackrel{f_{i}}\rightarrow K_i\rightarrow K_{i+1}[1]$ is an exact triangle such that $f_i$ is a spft morphism for $0\leq i\leq r-1$ with $K_0=H$ and $K_r=Q_r$, so $\mathrm{flatdim}_{\mathcal{T}}H\leq r$ by Theorem \ref{lem2.21}.
\end{proof}

 The following is the main theorem of this section.

\begin{thm}\label{lem5.5} Let $M$ be tilting and $\mathrm{add}(M)=\mathrm{add}({\i})$. One has equalities $$\mathrm{sup}\{\mathrm{flatdim}_{\mathcal{T}}X+\mathrm{inf}X\hspace{0.03cm}|\hspace{0.03cm}X\in\mathcal{T}^b\}
=\mathrm{wgldim}\mathcal{H}_M=\mathrm{gh.dim}\mathcal{T}
=\mathrm{add}(M)\textrm{-}\mathrm{tri.dim}\mathcal{T}^c.$$
\end{thm}
\begin{proof}  By Lemma \ref{lem2.11}, $\mathrm{sup}\{\mathrm{flatdim}_{\mathcal{T}}X+\mathrm{inf}X\hspace{0.03cm}|\hspace{0.03cm}X\in\mathcal{T}^b\}
\geq\mathrm{wgldim}\mathcal{H}_M=:d$. Conversely, let $X\in\mathcal{T}^{b}$. We show
$\mathrm{flatdim}_{\mathcal{T}}X\leq d-\mathrm{inf}X$  by the induction
on $\mathrm{amp}X$. The case of $\mathrm{amp}X=0$ is trivial as $\mathrm{flatdim}_{\mathcal{T}}H[n]=\mathrm{flatdim}_{\mathcal{T}}H+n$ for $H\in\mathcal{H}_M$ and $n\in\mathbb{Z}$. Assume that $\mathrm{amp}X>0$ and set $i=\mathrm{inf}X$. Consider the following  exact triangle $$\tau^{\geq i+1}(X)[-1]\rightarrow\mathcal{H}^{i}_M(X)[-i]\rightarrow X\rightarrow \tau^{\geq i+1}(X).$$ As $\mathrm{amp}\tau^{\geq i+1}(X)\leq\mathrm{amp}X-1$,  $\mathrm{flatdim}_{\mathcal{T}}\tau^{\geq i+1}(X)\leq d+\mathrm{inf}\tau^{\geq i+1}(X)$ by the induction. Note that  $\mathrm{flatdim}_{\mathcal{T}}X\leq\mathrm{max}\{\mathrm{flatdim}_{\mathcal{T}}\mathcal{H}^{i}_M(X)[-i],
\mathrm{flatdim}_{\mathcal{T}}\tau^{\geq i+1}(X)\}\leq d-i$ by Theorem \ref{lem2.1}, the first equality holds.

`$\mathrm{gh.dim}\mathcal{T}\leq\mathrm{wgldim}\mathcal{H}_M$'. Let $X\in\mathcal{T}^{c}$. By Lemma \ref{lem1.00}, one can choose an $M$-ghost triangle $X_{s+1}\rightarrow P_s\rightarrow X_s\rightarrow X_{s+1}[1]$ in $\mathcal{T}^c$ such that $P_s\in\mathrm{thick}^{1}(M)$ for $s\geq0$, where $X_0=X$. For $t\in\mathbb{Z}$, it follows from Lemma  \ref{lem2.0} that $$0\rightarrow\mathcal{H}^t_M(X_d)\rightarrow\mathcal{H}^t_M(P_{d-1})\rightarrow
\cdots\rightarrow\mathcal{H}^t_M(P_0)\rightarrow\mathcal{H}^t_M(X)\rightarrow0$$ is exact. As $\mathcal{H}^t_M(P_i)\in\mathcal{P}$ for $0\leq i\leq d-1$,  $\mathcal{H}^t_M(X_d)\in\mathcal{F}$, so $\mathcal{H}^{-t}_M(\mathbb{I}_{{\i}}(X_d))\cong\mathbb{I}_{{\i}}(\mathcal{H}^t_M(X_d))
\in\mathcal{I}$ for $t\in\mathbb{Z}$ by Lemma \ref{lem1.4}. Hence $\mathbb{I}_{{\i}}(X_d)\cong\prod_{t\in\mathbb{Z}}\mathcal{H}^{-t}_M(\mathbb{I}_{{\i}}(X_d))[t]
\in\langle\mathcal{I}\rangle_{1}$ by the dual of Lemma \ref{lem5.2'}, and then $\mathbb{I}_{{\i}}(X)\in\langle\mathcal{I}\rangle_{d+1}$. Let $X\xrightarrow{f_1} Y_1\xrightarrow{f_2}\cdots\xrightarrow{f_{n+1}} Y_{n+1}$ be a composite of $n+1$ $M$-ghosts. Then $\mathbb{I}_{{\i}}(f_{n+1}\cdots f_1)=\mathbb{I}_{{\i}}(f_{n+1})\cdots \mathbb{I}_{{\i}}(f_1)\in\mathcal{G}^{n+1}_M$ by Lemma \ref{lem1.2}, so $\mathbb{I}_{{\i}}(f_{n+1}\cdots f_1)=0$ by \cite[Lemma 4.5]{Y1}. Hence $f_{n+1}\cdots f_1$ is phantom. As $X$ is compact,  $f_{n+1}\cdots f_1=0$, and so $\mathrm{gh.len}_MX\leq d$, it implies that $\mathrm{gh.dim}\mathcal{T}\leq d$.

`$\mathrm{add}(M)\textrm{-}\mathrm{tri.dim}\mathcal{T}^c\leq\mathrm{gh.dim}\mathcal{T}$'. For $s\geq0$, the above exact triangle yields a commutative diagram:
\begin{center} $\xymatrix@C=30pt@R=20pt{
P_{s-1}[s-2] \ar@{=}[d]\ar[r]& X_{s-1}[s-2]\ar[d]\ar[r]^{\ f_{s}[s-2]} &X_{s}[s-1]\ar[d]\ar[r]&P_{s-1}[s-1]\ar@{=}[d] \\
P_{s-1}[s-2] \ar[r]& T_{s-1}\ar[d]\ar[r] &T_s\ar[d]\ar[r]&P_{s-1}[s-1] \\
     & X\ar[d]^{g_{s-1}}\ar@{=}[r]& X\ar[d]^{g_s}  \\
     & X_{s-1}[s-1]\ar[r]^{\ \ \ f_{s}[s-1]}& X_{s}[s]  }$
\end{center}where $T_1=P_0$ and $T_{d+1}\in\mathrm{thick}^{d+1}(M)$, $g_1=f_1$ and $f_s[s-1]\in\mathcal{G}_M$ for $s\geq1$. Since $\mathrm{gh.dim}_MX\leq d$ and $g_{d+1}=f_{d+1}[d]\cdots f_2[1]f_1\in\mathcal{G}^{d+1}_M$, it follows that $g_{d+1}$ is phantom. As $X$ is compact,
the exact triangle $$X_{n+1}[n]\longrightarrow T_{n+1}\longrightarrow X\stackrel{g_{n+1}}\longrightarrow X_{n+1}[n+1]$$ is split, so $X\in\mathrm{thick}^{d+1}(M)$. Thus $\mathrm{add}(M)\textrm{-}\mathrm{tri.dim}\mathcal{T}^c\leq d$.

`$\mathrm{wgldim}\mathcal{H}_M\leq\mathrm{add}(M)\textrm{-}\mathrm{tri.dim}\mathcal{T}^c$'. Let $\mathrm{add}(M)\textrm{-}\mathrm{tri.dim}\mathcal{T}^c=n$ and $X\in\mathcal{T}$. Then $X\simeq\mathrm{hcolim}_{\lambda\in\Lambda}X_\lambda$ with $X_\lambda\in\mathcal{T}^c$. Let $Y[-1]\rightarrow Z\rightarrow X\stackrel{\Phi}\rightarrow Y$  be an exact triangle  with $\Phi\in\mathcal{G}^{n+1}_M$, it induces the following commutative diagram of exact triangles\begin{center} $\xymatrix@C=20pt@R=20pt{
\triangle_\lambda:&Y[-1]\ar@{=}[d]\ar[r]& Z_\lambda\ar[d]\ar[r] &X_\lambda\ar[d]\ar[r]^{\Phi_\lambda}&Y\ar@{=}[d] \\
\triangle:&Y[-1]\ar[r]& Z\ar[r] &X\ar[r]^{\Phi\ }&Y}$
\end{center}Then $\Phi_\lambda\in\mathcal{G}^{n+1}_M$. As $X_\lambda\in\langle\mathcal{P}\rangle_{n+1}$, $\Phi_\lambda=0$ by \cite[Lemma 4.5]{Y1}. Hence $\Phi=\mathrm{hcolim}_{\lambda\in\Lambda}\Phi_\lambda$ is phantom by \cite[Lemma 2.8]{K1}. Fix a projective resolution
$$\cdots\rightarrow F_n\rightarrow F_{n-1}\rightarrow\cdots\rightarrow F_0\rightarrow H_0\rightarrow0$$ in $\mathcal{H}_M$ with the corresponding exact sequence $0\rightarrow H_{s+1}\rightarrow F_s\rightarrow H_s\rightarrow 0$. For $s\geq0$, one has a commutative diagram of exact triangles
\begin{center} $\xymatrix@C=30pt@R=20pt{
F_{s-1}[s-2] \ar@{=}[d]\ar[r]& H_{s-1}[s-2]\ar[d]\ar[r]^{\ h_{s}[s-2]} &H_{s}[s-1]\ar[d]\ar[r]&F_{s-1}[s-1]\ar@{=}[d] \\
F_{s-1}[s-2] \ar[r]& U_{s-1}\ar[d]\ar[r] &U_s\ar[d]\ar[r]&F_{s-1}[s-1] \\
     & H_0\ar[d]^{g_{s-1}}\ar@{=}[r]& H_0\ar[d]^{g_s}  \\
     & H_{s-1}[s-1]\ar[r]^{\ \ \ h_{s}[s-1]}& H_{s}[s]  }$
\end{center}where $U_1=F_0$, $g_1=f_1$ and $g_{n+1}=h_{n+1}[n]\cdots h_2[1]h_1\in\mathcal{G}^{n+1}_M$, it follows from the preceding proof that the triangle $$H_{n+1}[n]\longrightarrow U_{n+1}\longrightarrow H\stackrel{g_{n+1}}\longrightarrow H_{n+1}[n+1]$$ is pure. As $U\in\mathcal{F}\ast\mathcal{F}[1]\ast\cdots\ast\mathcal{F}[n]$, $\mathrm{fd}_{\mathcal{H}_M}U=\mathrm{flatdim}_{\mathcal{T}}U\leq n$ by Theorem \ref{lem2.21} and Lemma \ref{lem2.11}, it follows that
 $\mathrm{fd}_{\mathcal{H}_M}H=\mathrm{flatdim}_{\mathcal{T}}H\leq n$, as required.
\end{proof}

The theorem recovers the classical result of Hovey and Lockridge.

\begin{cor}\label{lem5.6} {\rm(\cite[Theorem 2.10]{H})} Let $R$ be a  ring and denote by $\mathrm{wgldim} R$ the weak global dimension of $R$.
Then $\mathrm{gh.dim}R=\mathrm{wgldim} R=\mathrm{add}(R)\textrm{-}\mathrm{tri.dim}\mathrm{D}^c(R)$.
\end{cor}

 This corollary is a generalization of \cite[Theorem 8.14]{EFHO} and \cite[Theorem 3.7]{L}.

\begin{cor}\label{lem4.0} If $\mathrm{add}(M)=\mathrm{add}({\i})$, then the following are equivalent:

$(1)$ $\mathrm{gh.dim}\mathcal{T}=0$;

 $(2)$ Any $M$-ghost triangle $X\rightarrow Y\rightarrow Z\stackrel{f}\rightarrow X[1]$  is pure;

 $(3)$ $M$ is tilting and $\mathcal{H}_M=\mathcal{F}$.
\end{cor}
\begin{proof} (1) $\Rightarrow$ (2) By \cite[Lemma 1.5]{K1},  $Z$ is a homology colimit of a filtered diagram of compact objects, i.e., $Z\simeq\textrm{hcolim}Z_\lambda$ for a diagram $\{Z_\lambda\}_{\lambda\in\Lambda}$ in $\mathcal{T}^c$. Let $C\in\mathcal{T}^c$ and $g:C\rightarrow Z$ be a morphism. As $\mathrm{Hom}_\mathcal{T}(C,Z)\cong\textrm{colim}\mathrm{Hom}_\mathcal{T}(C,Z_\lambda)$, $f$ factors through some $Z_\lambda$. One has a commutative diagram:\begin{center} $\xymatrix@C=20pt@R=20pt{
& & & C\ar[ddl]|g\ar[dl]|{g_\lambda}&\\
 X \ar@{=}[d]\ar[r]& Y_\lambda\ar[d]\ar[r] &Z_\lambda\ar[d]\ar[r]^{f_\lambda\ }&X[1]\ar@{=}[d] \\
X \ar[r]& Y\ar[r] &Z\ar[r]^{f\ \ }&X[1] }$
\end{center}As $f_\lambda\in\mathcal{G}_M=0$, $f_\lambda g_\lambda=0$, it implies that $fg=0$. Thus the triangle is pure.

(2) $\Rightarrow$ (3) As the exact triangle $\tau^{\leq-1}(M)\rightarrow M\rightarrow \mathcal{H}^0_M(M)\stackrel{f}\rightarrow\tau^{\leq-1}(M)[1]$ is $M$-ghost, it is pure, so $\mathcal{H}^0_M(M),\tau^{\leq-1}(M)\in\mathcal{F}$ by Lemma \ref{lem1.03}. Since $\mathbb{I}_{{\i}}(\tau^{\leq-1}(M))\in\mathcal{T}^{\geq1}$, $\mathrm{Hom}_\mathcal{T}(\mathbb{I}_{{\i}}(M),\mathbb{I}_{{\i}}(\tau^{\leq-1}(M)))\cong
\mathrm{Hom}_\mathcal{T}(\mathcal{H}^0_M(\mathbb{I}_{{\i}}(M)),\mathbb{I}_{{\i}}(\tau^{\leq-1}(M)))
=0$ by Lemma \ref{lem0.2}. Thus $\tau^{\leq-1}(M)=0$ and $M$ is tilting. Let $H\in\mathcal{H}_M$ and $K\rightarrow P\stackrel{f}\rightarrow H\stackrel{g}\rightarrow K[1]$ be an exact triangle with $f$ a sppj morphism. Then $g\in\mathcal{G}_M$ by Lemma \ref{lem2.0}, so $g$ is phantom. Thus $H\in\mathcal{F}$, and then $\mathcal{H}_M=\mathcal{F}$.

(3) $\Rightarrow$ (1) This follows from Theorem \ref{lem5.5}.
\end{proof}

\bigskip \centerline {\bf Acknowledgements}
\bigskip
This research was partially supported by National Natural Science Foundation of China (12571035, 12671053).

\end{document}